\documentclass{amsart}
\usepackage{graphics}
\usepackage{amsfonts,amsmath}
\usepackage{amsmath, amsfonts,amssymb,amsopn,amscd,amsthm}
 \usepackage{hyperref}
 \usepackage{mathrsfs}
 \hypersetup{colorlinks=true,linkcolor=blue,citecolor=blue,linktoc=page}

\begin{document}

 \newtheorem{thm}{Theorem}[section]
 \newtheorem{cor}[thm]{Corollary}
 \newtheorem{lem}[thm]{Lemma}{\rm}
 \newtheorem{prop}[thm]{Proposition}

 \newtheorem{defn}[thm]{Definition}{\rm}
 \newtheorem{assumption}[thm]{Assumption}
 \newtheorem{rem}[thm]{Remark}
 \newtheorem{ex}{Example}
\numberwithin{equation}{section}
\def\la{\langle}
\def\ra{\rangle}
\def\glexe{\leq_{gl}\,}
\def\glex{<_{gl}\,}
\def\e{{\rm e}}

\def\x{\mathbf{x}}
\def\P{\mathbf{P}}
\def\h{\mathbf{h}}
\def\by{\mathbf{y}}
\def\bz{\mathbf{z}}
\def\F{\mathcal{F}}
\def\R{\mathbb{R}}
\def\T{\mathbf{T}}
\def\N{\mathbb{N}}
\def\D{\mathbf{D}}
\def\V{\mathbf{V}}
\def\U{\mathbf{U}}
\def\K{\mathbf{K}}
\def\Q{\mathbf{Q}}
\def\H{\mathbf{H}}
\def\M{\mathbf{M}}
\def\oM{\overline{\mathbf{M}}}
\def\O{\mathbf{O}}
\def\C{\mathbb{C}}
\def\P{\mathbf{P}}
\def\Z{\mathbb{Z}}
\def\H{\mathcal{H}}
\def\A{\mathbf{A}}
\def\V{\mathbf{V}}
\def\AA{\overline{\mathbf{A}}}
\def\B{\mathbf{B}}
\def\c{\mathbf{C}}
\def\L{\mathcal{L}}
\def\bS{\mathbf{S}}
\def\H{\mathbf{H}}
\def\I{\mathbf{I}}
\def\Y{\mathbf{Y}}
\def\X{\mathbf{X}}
\def\G{\mathbf{G}}
\def\f{\mathbf{f}}
\def\z{\mathbf{z}}
\def\v{\mathbf{v}}
\def\y{\mathbf{y}}
\def\d{\hat{d}}
\def\bx{\mathbf{x}}
\def\bI{\mathbf{I}}
\def\y{\mathbf{y}}
\def\g{\mathbf{g}}
\def\w{\mathbf{w}}
\def\b{\mathbf{b}}
\def\a{\mathbf{a}}
\def\u{\mathbf{u}}
\def\q{\mathbf{q}}
\def\e{\mathbf{e}}
\def\s{\mathcal{S}}
\def\cc{\mathcal{C}}
\def\co{{\rm co}\,}
\def\tg{\tilde{g}}
\def\tx{\tilde{\x}}
\def\tg{\tilde{g}}
\def\tA{\tilde{\A}}

\def\supmu{{\rm supp}\,\mu}
\def\supp{{\rm supp}\,}
\def\cd{\mathcal{C}_d}
\def\cok{\mathcal{C}_{\K}}
\def\cop{COP}
\def\vol{{\rm vol}\,}
\def\om{\mathbf{\Omega}}
\def\blue{\color{blue}}
\def\red{\color{red}}
\def\blambda{\boldsymbol{\lambda}}
\def\bgamma{\boldsymbol{\gamma}}
\def\balpha{\boldsymbol{\alpha}}
\def\dis{\displaystyle}
\def\1{\boldsymbol{1}}
\def\va{\vert\balpha\vert}
\title{A Laplace duality for  integration}
\author{Jean B. Lasserre}
\address{LAAS-CNRS and Toulouse School of Economics (TSE)\\
LAAS, 7 avenue du Colonel Roche\\
31077 Toulouse C\'edex 4, France\\
Tel: +33561336415}
\email{lasserre@laas.fr}

\date{}

\maketitle
\begin{abstract}
We consider the integral $y\mapsto v(y)=\int_{K_y}f(\x)d\x$
on a domain $K_y=\{\x\in\R^d: g(\x)\leq y\}$, where
$g$ is continuous, nonnegative and $K_y$ is compact for all $y\in [0,+\infty)$.
Under some assumptions, we show that 
for every $y\in (0,\infty)$ there exists a distinguished scalar $\lambda_y\in (0,+\infty)$ such that
$v(y)=\int_{\R^d}f(\x)\exp(-\lambda_y\,g(\x))\,d\x$, which is 
the counterpart analogue for integration of Lagrangian
duality for optimization. A crucial ingredient is the Laplace transform, 
the analogue for integration of Legendre-Fenchel transform in optimization.
In particular, if 
both $f$ and $g$ are positively homogeneous then $\lambda_y$
is a simple explicitly rational function of $y$. In addition if 
$g$ is quadratic form then computing 
$v(y)$ reduces to computing the integral of $f$ 
with respect to a specific Gaussian measure for which exact and approximate numerical 
methods (e.g. cubatures) are available.\\
{\bf Keywords:}Integration, Laplace transform, optimization, positively homogeneous functions,
\end{abstract}

\section{Introduction}

Let $g:\R^d\to\R$ be continuous, nonnegative and such that
\begin{equation}
\label{set-Ky}
K_y\,:=\,\{\,\x\in\R^d: \: g(\x)\,\leq\,y\,\}\,,\quad y\in [0,+\infty)\,.
\end{equation}
is compact for all $y\in [0,+\infty)$, and let $f:\R^d\to\R$ be continuous. Consider the function
\begin{equation}
\label{def-v}
y\mapsto v(y)\,:=\,\int_{K_y}f(\x)\,d\x\,,\quad y\in [0,+\infty)\,,
\end{equation}
which is well-defined for every $y\in [0,+\infty)$. This problem appears in several 
area of science and engineering.
For instance, if $f$ the density of a Gaussian measure $\mu$ then $v(y)=\mu(K_y)$ 
provides the probability that a Gaussian  random vector (with distribution $\mu$) lies in $K_y$, 
a basic problem encountered in probability, analysis of dynamical systems, and 
space engineering. In particular, 
in aerospace engineering one encounters integrals \eqref{def-v}
(where $f=\exp(h)$ with both $h$ and $g$ quadratic polynomials)
for computing probability of collision for satellites;
the interested reader is referred to \cite{collision}, and the recent survey
\cite {survey}. Next, as we will see in the sequel,
the asymptotics of $v$ as $y\to 0$ or
$y\to\infty$, is related to the asymptotics of the exponential integral $\int f\,\exp(-\lambda\,g(\x))\,d\x$
as $\lambda\to \infty$ or $\lambda\to 0$. Under assumptions on $f$ and $g$, deep results 
on the latter exponential integral invoke Newton's diagram and classification of \emph{minimal points},
as discussed in \cite{arnold,varchenko,vasiliev-1,vasiliev-2} and references therein.

\subsection*{Contribution}
(i) Assuming that $f$ is nonnegative and the Laplace transform of $v$ exists 
(see e.g. \cite[\S 6.26.1]{Laplace-book} and satisfies the \emph{Final Value Theorem}
(see e.g. \cite[Theorem 3.8.2]{Laplace-book}), our first contribution  is to show that 
for every $y\in (0,+\infty)$ there exists $\lambda_y\in (0,+\infty)$ such that:
\begin{equation}
\label{main-1}
v(y)\,=\,\int_{\R^d}f(\x)\,\exp(-\lambda_y\,g(\x))\,d\x\,.
\end{equation}
That is,  integrating $f$ on $K_y$ reduces to integrating $f$ 
\emph{on the whole space} $\R^d$ but now against the measure with density $\x\mapsto \exp(-\lambda_y\, g(\x))$ with respect to Lebesgue measure. 
In some cases such integrals on the whole $\R^d$ can be approximated, e.g. by cubatures \cite{Cools-1,Cools-2}.
For instance if $g$ is convex quadratic, then specific cubatures \cite{cubature} are available for the Gaussian density 
$\x\mapsto \exp(-\lambda_y\, g(\x))$ (but also for direct integration on the ellipsoid $K_y$).
On the other hand, depending on $g$, integration on nasty domains $K_y$ can be 
quite difficult, e.g. if $g$ is not convex as in
Example \ref{example-1} in \S \ref{sec:main}, while for \eqref{main-1} one may use
cubatures (e.g. Gauss-Legendre) on a large box $[-r,r]^d$ 
to integrate the function $f\,\exp(-\lambda\,g)$ which is continuous;
see the discussion in Remark \ref{rem-1}.

(ii) Next, if $g$ and $f$ are positively homogeneous of degree $d_g$ and $d_f$ respectively, then $y$ and $\lambda_y$ are related by
\begin{equation}
\label{link}
y\cdot\lambda_y\,=\,\left(\Gamma(1+(d+d_f)/d_g)\right)^{d_g/(d+d_f)}\,,
\end{equation}
for every $y\in (0,+\infty)$. In particular, let $f$ be a polynomial 
of degree $d_f$, and write $f=\sum_{k=0}^{d_f}f_k$ where for each $k$,
$f_k$ homogeneous of degree $k$. Then 
for every $y\in (0,+\infty)$, 
\[v(y)\,=\,\sum_{k=0}^{d_f}\int_{\R^d}f_k(\x)\,\exp(-\lambda_{y,k}\,g(\x))\,d\x\,,\]
where for every $k=0,1,\ldots,d_f$,
\[\lambda_{y,k}\,=\,\frac{\left(\Gamma(1+(d+k)/d_g)\right)^{d_g/(d+k)}}{y}\,,\quad y\in (0,+\infty)\,.\]

(iii) Finally, we interpret \eqref{main-1} as a \emph{duality} result, namely a 
duality analogue for integration (hence in the usual $(+,\times)$-algebra) of 
\emph{Lagrangian duality} for optimization 
(hence in the $(\max,+)$-algebra).  
 Indeed, associated with the optimization problem
\begin{equation}
\label{hat-v}
\P:\quad \hat{v}(z)\,=\,\inf\,\{\,f(\x):\:h(\x)\,\geq\,z\,\}\,,\quad z\in\R\,,
\end{equation}
(with feasible set $\tilde{K}_z:=\{\x: h(\x)\geq z\}$) is the Lagrangian 
\begin{equation}
\label{lagrangian-optim}
\x\mapsto L(\x,\lambda)\,:=\,f(\x)-\lambda\,h(\x)\,.
\end{equation}
(With $h:=-g$ and $z=-y$, one retrieves $K_y$ in \eqref{set-Ky}.)

If $f$ and $-h$ are  convex 
then so is $\hat{v}$ and its Legendre-Fenchel transform $\hat{v}^*$ reads:
\begin{eqnarray}
\label{Fenchel}
\hat{v}^*(\lambda)&=&
\sup_y \lambda\,y-\hat{v}(y)\\
\nonumber
&=&\left\{\begin{array}{l}
-\displaystyle\inf_\x\,f(\x)-\lambda\,h(\x)\:\mbox{if $\lambda\geq0$,}\\
+\infty\:\mbox{otherwise.}\end{array}\right.
\end{eqnarray}
Then applying Legendre-Fenchel to $\hat{v}^*$ yields
\begin{eqnarray}
\label{Fenchel-again}
\hat{v}(z)&=&\sup_{\lambda}\lambda\,z-\hat{v}^*(\lambda)\\
\nonumber
&=&\displaystyle\
\sup_{\lambda\geq0}\,\lambda\,y+G(\lambda)\\
\label{min-lagrangian}
\mbox{with $G(\lambda)$}&:=&\inf_\x L(\x,\lambda)\,,\quad\lambda\geq0\,.
\end{eqnarray}
Under convexity assumptions on $f$ and $-h$, there exist 
a maximizer $\x_z\in \tilde{K}_z$ and KKT-multiplier $\lambda^*_z\geq0$ such that
\[\nabla f(\x_z)-\lambda^*_z\nabla h(\x_z)\,=\,0\,;\quad \lambda^*_z\,(h(\x_z)-z)\,=\,0\,,\]
and so
\begin{equation}
\label{equiv}
\hat{v}(z)\,=\,\lambda^*_z\,z+\inf_{\x}f(\x)-\lambda^*_z\,h(\x)\,,\quad \forall z\,.
\end{equation}
So the original minimization of $f$ on $\tilde{K}_z$ reduces to 
the minimization of  the Lagrangian $L(\x,\lambda^*_z)$ now over the whole $\R^d$,
for a distinguished value $\lambda^*_z$ of the KKT-multiplier $\lambda\geq0$ associated with the constraint $h(\x)\geq z$.

In particular, if $h$ (resp. $f$) is differentiable and positively homogeneous of degree $d_h$ (resp. $d_f$), then using Euler's identity $\langle\x_z,\nabla f(\x_z)\rangle =d_f\,f(\x_z)$ 
(resp. $\langle\x_z,\nabla h(\x_z)\rangle =d_h\,h(\x_z)$), yields
$\lambda_z\cdot z= \hat{v}(z)\,d_f/d_h$ (to compare with \eqref{link}).

-- The analogue for integration of the Lagrangian $L(\x,\lambda)$ in \eqref{min-lagrangian}
for optimization, is the ``Lagrangian" integrand 
$\x\mapsto f(\x)\exp(-\lambda\,g(\x))$, $\x\in\R^d$, and 

-- the analogue of the dual function $G(\lambda)$
in \eqref{min-lagrangian} is just the integral 
\[\int_{\R^d}f(\x)\exp(-\lambda\,g(\x))\,d\x\,\]
(where the ``$\max_{\x\in\R^d}$" has been  replaced with ``$\int_{\x\in\R^d}$").

-- Finally, the analogue for integration of the 
Legendre-Fenchel transform \eqref{Fenchel}
in optimization, is the Laplace transform 
\[\L_v(\lambda)\,=\,\int_0^\infty \exp(-\lambda\,y)\,v(y)\,dy\,,\quad\forall\lambda\in \C\,,\Re (\lambda)>0\,,\]
of $v$ in \eqref{def-v}, and the  analogue of \eqref{Fenchel-again} is 
the inverse Laplace transform 
\[v(y)\,=\,\frac{1}{2\pi i}\int_{c-i\infty}^{c+i\infty}\exp(\lambda y)\,\L_v(\lambda)\,d\lambda\,,\quad\forall y\in (0,+\infty)\,,\]
(with $c>\Re(a)$ for some $a\in \C$).

Formal analogies between concepts in optimization and their counterparts in integration are not new and have been observed in a number of domains. For instance, \emph{convolution} of Gaussian distributions in probability is the analogue 
of \emph{inf-convolution} of quadratic forms in optimization. Similarly, concepts in probability have their counterparts in Dynamic Programming  as outlined and described in \cite[\S 9.4]{linearity}. 
Finally for convex polytopes $K_y\subset\R^d$, in \cite{lasserre-book} one has shown \emph{explicit} links between \emph{integration} and \emph{counting} on the one hand, and linear (LP) and integer programming (IP) on the other hand. In particular, classical LP ingredients (basis, vertex, reduced gradient, and dual vector) also appear explicitly  in Brion \& Vergne's formula for integration and counting over convex polytopes; for more details see \cite{lasserre-book} and references therein.

The present contribution is also in the same spirit as  in \cite{lasserre-book} and 
\cite{lass-zeron} but now for
integration of a larger class of functions (continuous rather than linear in \cite{lasserre-book}) 
on a larger class of domains ($K_y$ instead of convex polytopes in \cite{lasserre-book}, or very specific 
domains\footnote{For several cases studied in \cite{lass-zeron},
$v(y)$ can be obtained from straightforward inversion of $\mathcal{L}_v$.} in \cite{lass-zeron}). Indeed, 
an integration domain of the form $\{\,\x\in\R^d: g_j(\x)\leq b_j,\:j=1,\ldots,m\,\}$  reduces to
$\{\,\x\in\R^d: \max_j \tilde{g}_j(\x)\leq 1,\:j=1,\ldots,m\,\}$ with
$\tilde{g}_j(\x):=g_j(\x)/b_j$). Then one considers the set 
$K_y:=\{\,\x\in\R^d: \max_j \tilde{g}_j(\x)\leq y\,\}$. (In particular notice
that if the $g_j$'s are all positively homogeneous of same degree $d_g$,
then so is the function $\x\mapsto \max_j g_j(\x)$.) Crucial in all the references 
\cite{lass-zeron,lasserre-book} is to embed a specific integral on $K_1$ 
in a larger parametrized family of integrals on $K_y$, with values $v(y)$, $y\in (0,+\infty)$, and then apply the Laplace transform to $v$.  This is exactly what is done in optimization on $K_y$ where
one  applies the Legendre-Fenchel transform to the \emph{value function} $\hat{v}(y)$ in \eqref{hat-v}.
When the Laplace transform has a closed form expression then its inverse (e.g. for $y=1$ to obtain $v(1)$) can sometimes be 
computed efficiently (see e.g. \cite{lass-zeron});  in optimization it corresponds 
to the case where the minimization of the Lagrangian $L(\x,\lambda)$ with respect to $\x$
in \eqref{min-lagrangian},  provides $G(\lambda)$  explicitly.

\section{Main result}

\subsection{Notation and definitions}

Let $\R[\x]=\R[x_1,\ldots,x_d]_n$  be the ring of polynomials 
in the real variables $x_1,\ldots,x_d$, of degree at most $n$. 
Denote by $\R_+\subset\R$ the positive half-line.
A function $f$ is positively homogeneous of degree $d_f$ if 
$f(\lambda\,\x)=\lambda^{d_f}f(\x)$ for all $\x\in\R^d$ and all $\lambda>0$.
When $f$ is continuously differentiable, Euler's identity states that
$\langle\nabla f(\x),\x\rangle=d_f\,f(\x)$, $\forall \x\in\R^d$.

\subsection*{Laplace transform}
For a function $h:\R_+\to\R$,  its Laplace transform $\L_h:\C\to\C$, is defined by:
\begin{equation}
 \label{laplace-def}
 \L_h(\lambda)\,=\,\int_0^\infty h(y)\,\exp(-\lambda\,y)\,dy\,,\quad \lambda\in\C\,,
\end{equation}
provided that the integral is well-defined. For instance a sufficient condition is 
that  $h$ is of \emph{exponential order} $\exp(at)$ ($a>0$) as $t\to\infty$,
i.e. there exists $T,M>0$ such that for all $t>T$, $\vert h(t)\vert\leq M\exp(at)$; see e.g. \cite[\S 3.3]{Laplace-book}. 
If $h$ is continuous or piecewise continuous in every finite interval $(0,T)$,
then $\L_h$ exists for all $\lambda\in \C$ with $\Re(\lambda)>a$.
When it is the case then we can recover $h$ via the inverse Laplace transform:
\[h(y)\,=\,\frac{1}{2\pi i}\int_{c-i\infty}^{c+i\infty} \exp(\lambda\,y)\,\L_h(\lambda)\,d\lambda\,,\quad y\in (0,+\infty)\,,\]
where $c> a$.
We also have the well-known \emph{Initial Value Theorem} (\cite[Theorem 3.8.1]{Laplace-book})
\begin{equation}
\label{initial-value}
\lim_{\lambda\to \infty}\lambda\,\L_h(\lambda)\,=\,\lim_{y\to 0}h(y)\,,
\end{equation}
and if $\lim_{y\to\infty}  h(y)$ exists, the \emph{Final Value Theorem} 
\begin{equation}
\label{final-value}
\lim_{\lambda\to 0}\lambda\,\L_h(\lambda)\,=\,\lim_{y\to \infty}h(y)\,.
\end{equation}
The latter holds under additional assumptions; see \cite[Theorem 3.8.2, pp. 110--112]{Laplace-book}.

\subsection{Main result}
\label{sec:main}
\begin{assumption}
\label{ass-1}
(i) The function $g:\R^d\to\R$ is continuous, nonnegative and the set
$K_y$ in \eqref{set-Ky} is compact for every $y\in [0,\infty)$.

(ii) The function $f:\R^d\to\R$ is continuous and nonnegative.
\end{assumption}
\vspace{0.1cm}

Let $v:[0,\infty)\to\R$ be as in \eqref{def-v}. 
Observe that if one knows how to evaluate $v(y)$ for continuous nonnegative functions,
then we may also evaluate $v(y)$ for continuous functions bounded from below.
Indeed if $f$ is bounded below,
say $f\geq \tau$ for some $\tau>-\infty$, then
\[v(y)\,=\,\int_{K_y}f\,d\x\,=\,\tau\,\int_{K_y}1\,d\x+\int_{K_y}(f-\tau)\,d\x\]
i.e. $v$ is a weighted sum of two integrals of nonnegative functions ($1$ and $f-\tau)$). Therefore it suffices to restrict to the family of nonnegative functions $f$. 

\begin{thm}
\label{thm-laplace}
Let Assumption \ref{ass-1} hold. 

(i) Assume that $v$ in \eqref{def-v} is of exponential order $\exp(at)$ for every $a>0$\footnote{For example if the growth of $v$ is at most polynomial in $y$.}. Then 
for every real $\lambda>0$,
\begin{equation}
\label{thm-laplace-1}
\L_{v}(\lambda)\,=\,\frac{1}{\lambda}\,\int_{\R^d}f(\x)\,\exp(-\lambda\,g(\x))\,d\x\,,
\end{equation}
(ii) In addition, assume that the set $K_0=\{\,\x\in\R^d: g(\x)=0\,\}$ has Lebesgue measure zero, and that the Final Value Theorem \eqref{final-value} holds (possibly with $+\infty$ limit). Then
for every $y\in (0,+\infty)$, there exists $\lambda_y\in (0,+\infty)$ such that
\begin{equation}
\label{y-lambda}
v(y)\,=\, \lambda_y\,\L_v(\lambda_y)\,=\,\int_{\R^d}f(\x)\,\exp(-\lambda_y\,g(\x))\,d\x\,.
\end{equation}
Similarly,  if $v$ is continuous\footnote{As pointed out by a referee,
$v$ may indeed be discontinuous. If $g$ is a nontrivial polynomial
then $v$ is continuous.} then for every $\lambda\in (0,+\infty)$ there exists $y_\lambda\in (0,+\infty)$ such that
\begin{equation}
\label{lambda-y}
v(y_\lambda)\,=\,\lambda\,\L_v(\lambda)\,=\,\int_{\R^d}f(\x)\,\exp(-\lambda\,g(\x))\,d\x\,.
\end{equation}
\end{thm}
\begin{proof}
(i) By \cite[Theorem 3.31]{Laplace-book}, $\L_v$ exists for all $\lambda$ provided that $\Re(\lambda)>0$. Next, let $\lambda\in\R$. Then
 \begin{eqnarray*}
 \L_v(\lambda)&=&\int_0^\infty v(y)\,\exp(-\lambda\,y)\,dy\\
 &=&\int_0^\infty \left(\int_{K_y}f(\x)\,d\x\right)\,\exp(-\lambda\,y)\,dy\\
 &=&\int_{\R^d} f(\x) \left(\int_{g(\x)}^\infty \exp(-\lambda\,y)\,dy\,\right)\,d\x\\
 &=&\frac{1}{\lambda}\int_{\R^d}f(\x)\,\exp(-\lambda\,g(\x))\,d\x\,,\quad\forall \lambda\,>\,0\,,
\end{eqnarray*}
which yields \eqref{thm-laplace-1}. (The third equality is obtained by 
a standard Fubini-Tonelli interchange.)

(ii) In view of the assumptions on $f$ and $g$, the function $v$ is non decreasing, and the Initial Value Theorem \eqref{initial-value} holds. Hence  
$\lim_{\lambda\to\infty}\lambda\,\L_v(\lambda)=v(0)=0$. Next, as the Final value Theorem holds 
we also have $\lim_{\lambda\to 0}\lambda\,L_v(\lambda)=\lim_{y\to\infty}v(y)=:v(\infty)$ with
possibly $v(\infty)=+\infty$. Moreover, the function
\[\lambda \mapsto \lambda\,\L_v(\lambda)\,=\,\int_{\R^d}f(\x)\,\exp(-\lambda\,g(\x))\,d\x\,,\quad\lambda\in (0,+\infty)\,,\]
is continuous and therefore,
\[\lambda\,\L_v(\lambda)(0,+\infty)\,=\,(0,v(\infty))\,\supseteq\,v((0,\infty))\,,\]
which yields the desired result \eqref{y-lambda}, and also yields \eqref{lambda-y} if $v$ continuous.
\end{proof}
Hence Theorem \ref{thm-laplace} states that for every $y\in (0,+\infty)$, integrating $f$ on $K_y$ w.r.t. Lebesgue measure is the same as integrating $f $ on the whole space $\R^d$ but now against the measure 
with density $\exp(-\lambda_y\,g(\x))$ w.r.t. Lebesgue measure, for some distinguished real scalar $\lambda_y\in (0,+\infty)$. 

\begin{rem}
\label{rem-1}
For a fixed $\lambda$, evaluating the integral
$\int f\exp(-\lambda\,g)d\x$ is challenging but numerical approximations are sometimes available, e.g. 
via cubature formula for the weight function $\x\mapsto \exp(-\lambda_y\,g(\x))$
for some functions $g$; see e.g. \cite{Cools-1,Cools-2}. A prototypal and important case is when 
$g$ is a (nonnegative) quadratic polynomial. Then the measure $\exp(-\lambda_y\,g(\x))\,d\x$
is (up to scaling) a Gaussian measure for which numerical integration techniques are is well documented; see e.g. \cite{cubature}; however, $v(y)$ can be also computed directly by integrating $f$
via efficient cubatures on the ellipsoid $K_y$. On the other hand,  when $g$ is homogeneous nonnegative 
(and possibly not convex) then $K_y$ can have a nasty geometry. But the integral on $\R^d$ 
in \eqref{y-lambda} may still
be well approximated, e.g. by integration of the \emph{continuous} function $f\exp(-\lambda g)$ on a
box $\mathbf{B}_r=[-r,r]^d$ (with $r$ sufficiently large) via cubatures w.r.t. Lebesgue measure on $\mathbf{B}_r$ (e.g. Gauss-Legendre cubatures).
In contrast computing $v(y)$ in \eqref{def-v} as
$\int_{K_y}fd\x=\int_{[-r_0,r_0]^d}f\,1_{K_y}d\x$ (now on the box $\mathbf{B}_{r_0}\supset K_y$)  via similar cubatures 
may not be precise because the integrand $f(\x)\,1_{K_y}(\x)$ is \emph{not} continuous
and may be nasty. 
\end{rem}
\begin{ex}
\label{example-1}
Examples of nasty compact sets $K_y$ for homogeneous polynomials $g$ are displayed in 
Figure \ref{figure-1}.
\begin{figure}[ht]
\resizebox{0.7\textwidth}{!}
{\includegraphics{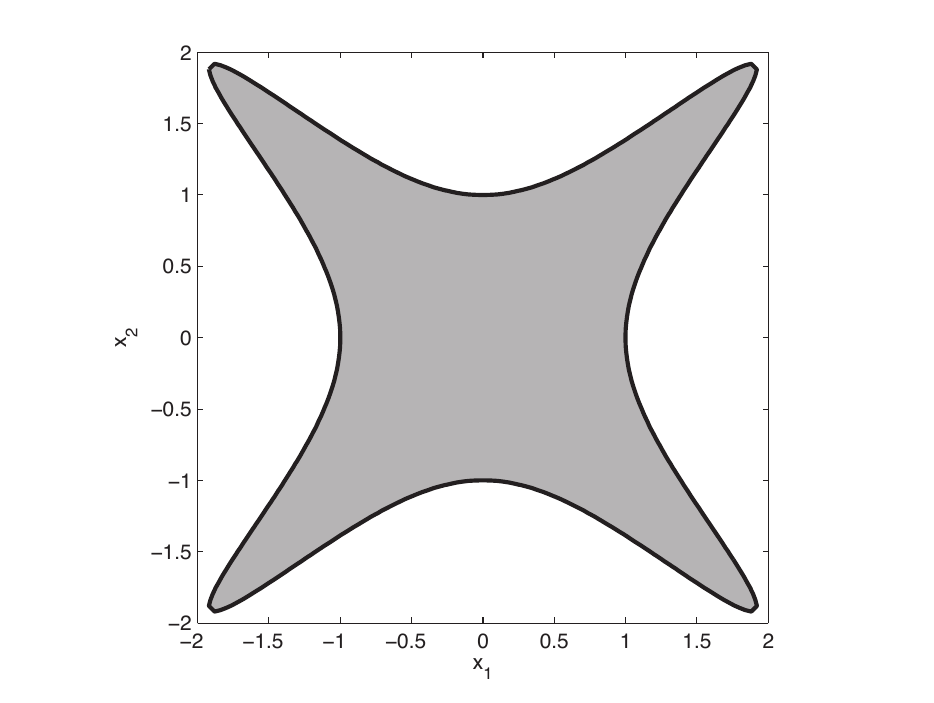}\includegraphics{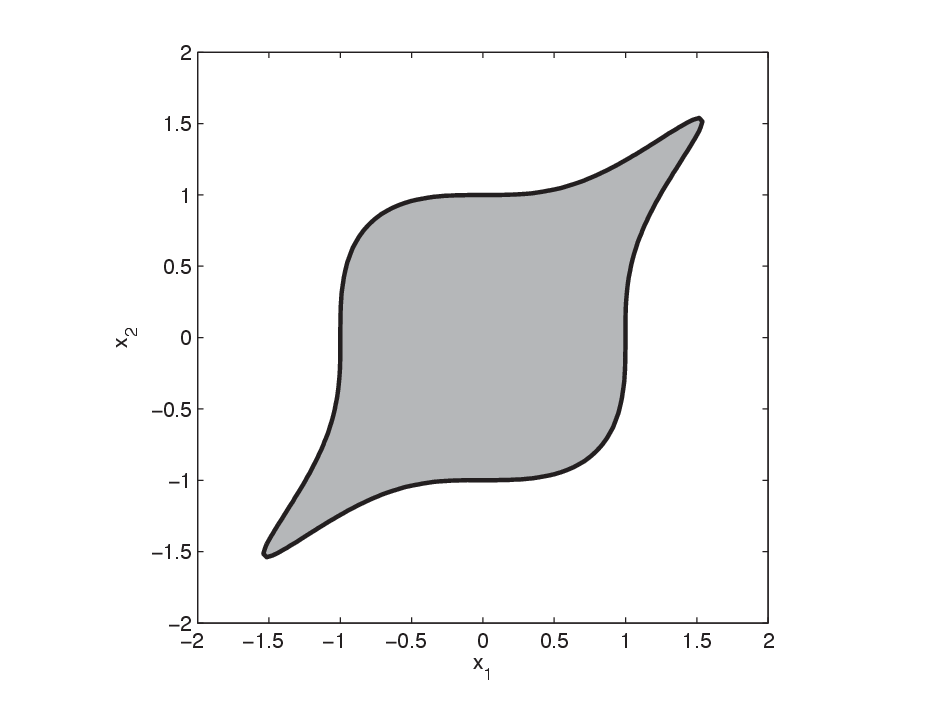}}
\caption{$K_1$ with $g(\x)=x^4+y^4-1.925\,x^2y^2$ (left) and with $g(\x)=x^6+y^6-1.925\,x^3y^3$ (right)\label{figure-1}}
\end{figure}
\end{ex}
Equation \eqref{y-lambda} is also the analogue for integration of
the Lagrangian relaxation \eqref{equiv} in optimization, where under some convexity assumptions,
the minimization of $f$ on $\tilde{K}_z=\{\x: h(\x)\geq z\}$ is replaced with the minimization of 
the Lagrangian $\x\mapsto L(\x,\lambda^*_z)=f-\lambda^*_z\,h$ on the whole $\R^d$ for some distinguished multiplier $\lambda^*_z\geq0$.
Indeed with $z$ fixed, the KKT-optimality conditions
at a local minimizer $\x^*_z\in \tilde{K}_z$ of $\P$ state that there exists $\lambda^*_z\geq0$
such that  $\x^*_z$ is a critical point of the Lagrangian  $L(\x,\lambda^*_z)$;  in addition, 
if $f$ and $-h$ are convex then $\x^*_z$ is a global minimizer of $L(\x,\lambda^*_z)$.

Next, with $\lambda$ fixed and $\L_v$ the Laplace transform of $v$,
\begin{equation}
\label{equiv-2}
\lambda\,\L_v(\lambda)
\,=\,\int_{\R^d}f(\x)\,\exp(-\lambda\,g(\x))\,d\x\,,\end{equation}
is the counterpart for integration of the Legendre-Fenchel transform of $\hat{v}$ in \eqref{hat-v}:
\begin{eqnarray}
\nonumber
\hat{v}^*(\lambda)&=&\sup_{z}\lambda\,z-\hat{v}(z)\,=\,-\inf_{\x\in\R^d} f(\x)-\lambda\,h(\x)\\
\label{equiv-3}
&=&\sup_{\x\in\R^d} \:\lambda\,h(\x)-f(\x)\,
\quad\lambda\geq0\,.\end{eqnarray}

In \eqref{equiv-2} the infinitesimal sum ``$\int_{\R^d}$"  is the analogue of
 ``$\sup_{\x\in\R^d}$" in \eqref{equiv-3}.  So it is fair to consider 
the integrand $\x\mapsto f(\x)\,\exp(-\lambda\,g(\x))$  in \eqref{equiv-2}
as the counterpart of the Lagrangian  $L(\x,\lambda)$
in optimization \eqref{lagrangian-optim}. Moreover, we call the distinguished scalar $\lambda_y$ in Theorem \ref{thm-laplace} a \emph{``Laplace dual variable"},
the exact analogue for integration of the Karush-Kuhn-Tucker Lagrange multiplier $\lambda^*_z$ in \eqref{equiv} for optimization, associated with the constraint $h(\x)\geq z$.

These correspondences
provide with an additional instance of formal analogies between duality in optimization and integration
via Legendre-Fenchel and Laplace transforms respectively, in the spirit of 
those investigated in \cite{lasserre-book} for LP and Integer Programming on the one hand and
linear integration and counting on the other hand.

However, so far Theorem \ref{thm-laplace} is only a qualitative result as it does not provide 
a clue on what is the scalar $\lambda_y$ associated with $y\in (0,+\infty)$. We next 
partially address this issue under additional assumptions on $f$ and $g$.
\begin{cor}
\label{cor-homog}
Let Assumption \ref{ass-1}(i) holds. In addition
let $g$ (resp. $f$) be positively homogeneous of degree $d_g$ (resp. $d_f$). 
Then with $v$ as in \eqref{def-v}:
\begin{equation}
\label{lem-laplace-1}
\L_{v}(\lambda)\,=\,\frac{1}{\lambda^{1+(d+d_f)/d_g}}\,\int_{\R^d}f(\x)\,\exp(-g(\x))\,d\x\,,
\end{equation}
for all $\lambda\in\,(0,+\infty)$, and
\begin{equation}
\label{lem-laplace-2}
v(y)\,=\,\frac{y^{(d+d_f)/d_g}\int_{\R^d}f(\x)\,\exp(-g(\x))\,d\x}{\Gamma(1+(d+d_f)/d_g)}\,
\end{equation}
for all $y\in [0,+\infty)$.
In addition, for every $y\in (0,+\infty)$, 
\begin{equation}
 \label{lem-laplace-3}\,
 v(y)\,=\,\lambda_y\, \L_v(\lambda_y)\,=\,\int_{\R^d}f(\x)\,\exp(-\lambda_y\,g(\x))\,d\x\,,
  \end{equation}
with $y\,\lambda_y=\Gamma(1+(d+d_f)/d_g)^{d_g/(d+d_f)}$.
 \end{cor}
 \begin{proof}
 By \eqref{thm-laplace-1} in Theorem \ref{thm-laplace}(i), for every real scalar $\lambda>0$
 (and using that $f,g$ are positively homogeneous),
 \begin{eqnarray*}
 \L_v(\lambda)&=&\int_{\R^d}f(\x)\,\exp(-\lambda\,g(\x))\,d\x\,,\quad\Re(\lambda)>0\\
 &=&\frac{1}{\lambda^{1+(d+d_f)/d_g}}\int_{\R^d}f(\x)\,\exp(-g(\x))\,d\x\,,
  \end{eqnarray*}
 which yields \eqref{lem-laplace-1}. On the other hand, as $f$ and $g$ are positively homogeneous,
 $v$ is also positively homogeneous of degree $(d+d_f)/d_g)$, and $v$ being univariate, 
  \begin{eqnarray*}
  v(y)&=&y^{(d+d_f)/d_g}\,v(1)\quad  \Rightarrow\\
  \L_v(\lambda)&=&\frac{\Gamma(1+(d+d_f)/d_g)}{\lambda^{1+(d+d_f)/d_g)}}\,v(1)\,,
  \quad\Re(\lambda)>0\,,
  \end{eqnarray*}
 from which we deduce that $v(1)\Gamma(1+(d+d_f)/d_g)=\int_{\R^d}f(\x)\exp(-g(\x))\,d\x$,
 and from  which \eqref{lem-laplace-2} follows.
 So with $y\in (0,+\infty)$, 
 \begin{eqnarray*}
 \lambda_y\,\L_v(\lambda_y)&=&v(y)  \quad\Leftrightarrow\\
 \lambda_y\,y&=&
 \Gamma(1+(d+d_f)/d_g)^{d_g/(d+d_f)}\,,\end{eqnarray*}
 which yields \eqref{lem-laplace-3}.
\end{proof}
 \vspace{.2cm}
 So Corollary \ref{cor-homog} identifies the Laplace (or ``dual") variable $\lambda_y$ 
 associated with each $y\in (0,+\infty)$, and such that 
 integrating $f$ on $K_y$ reduced to integrating $f$ 
 on $\R^d$, but now against the measure with density $\exp(-\lambda_y\,g(\x))$ 
 w.r.t. Lebesgue measure. 
 
 Of course, in view of \eqref{lem-laplace-2}, to evaluate $v(y)$ is suffices to evaluate 
 $v(1)$, or equivalently, to compute the single integral 
 $\int_{\R^d}f(\x)\,\exp(-g(\x))\,d\x$, which is quite a difficult task in general. 
 However, to approximate the integral,  one may invoke numerical tools 
 like cubatures for the weight function $\exp(-g(\x))$  as in e.g. \cite{Cools-1,Cools-2}; 
 when $K_y$ has a nasty geometry then
 integrating the continuous function $f(\x)\exp(-g(\x))$ on $\R^d$ via cubatures 
 w.r.t. Lebesgue measure on the box $\mathbf{B}_r:=[-r,r]^d$ for large $r$, might be preferable to
 integrating directly $f$ on $K_1$, e.g., via cubatures w.r.t. Lebesgue measure on a box $\mathbf{B}_{r_0}\supset K_1$ because the function $f(\x)\,1_{K_1}(\x)$ is \emph{not} continuous; see Remark \ref{rem-1}.  An even more specific case is when $g$ is a nonnegative quadratic form in which case one has to integrate $f$ against a Gaussian measure for which 
 several specialized procedures exist \cite{cubature}.
 Notice also that in the particular case where $d+d_f=d_g$ then $\lambda_y\,y=1$ and
 so $v(y)=\int_{\R^d}f(\x)\,\exp(-g(\x)/y)\,d\x$ for all $y\,\in (0,+\infty)$.
 
\subsection*{The case of polynomials}
  When $f$ is a polynomial, then write $f$ as  $f(\x)=\sum_{k=0}^{d_f}f_k(\x)$,
 where for each $k$, $f_k$ is a homogeneous polynomial of degree $k$,
 and let
 \[v_k(y)\,:=\,\int_{K_y}f_k(\x)\,d\x\,,\quad k=0,\ldots,d_f\,.\]
In view of Corollary \ref{cor-homog},
 for every $y\in (0,+\infty)$:
  \begin{eqnarray}
  \label{eq:poly}
  v(y)&=&\sum_{k=0}^{d_f}\int_{\R^d}f_k(\x)\,\exp(-\lambda_{y,k}\,g(\x))\,d\x\\
  \label{eq:poly-2}
  &=&\sum_{k=0}^{d_f}\lambda_{y,k}\,\mathcal{L}_{v_k}(\lambda_{y,k})\,,
 \end{eqnarray}
with $\lambda_{y,k}=\Gamma(1+(d+k)/d_g)^{d_g/(d+k)}/y$.
Hence in this case, the explicit description \eqref{eq:poly}
of $v$ in terms of integrals on the whole $\R^d$, follows from homogeneity with no need of $f$ to be nonnegative.

\subsection*{The case of the simplex}
Let $\e:=(1,\ldots,1)\in\R^d$ and $K_y:=\{\,\x\geq0: \e^T\x\leq y\}$ (a dilation of the canonical simplex).
As one integrates over  
a subset of $\R^d_+$ one may and will assume that $f(\x)=0$ on 
$\R^d\setminus\R^d_+$. Then: 
\begin{eqnarray*}
\L_v(\lambda)&=&\int_0^\infty v(y)\exp(-\lambda\,y)\,dy\\
&=&\frac{1}{\lambda}\int_{\R^d_+}f(\x)\,\exp(-\langle\lambda\e,\x\rangle)\,d\x\,=\,
\frac{1}{\lambda}\,\hat{\L}_f(\lambda\,\e)\,,
\end{eqnarray*}
for all real $\lambda$ such that the integral is finite, and where  $\hat{\mathcal{L}}_f:\C^d\to\C$ is the multivariate Laplace transform of $f$, that is,
\[\hat{\mathcal{L}}_f(\bgamma)\,:=\,\int_{\{\x:\x\geq0\}} f(\x)\,\exp(-\langle\bgamma,\x\rangle)\,d\x\,,\]
where the integral is defined for $\Re(\bgamma)>\a$ for some $\a\in\R^d$.  So
the univariate Laplace transform $\mathcal{L}_v$ of $v$ evaluated at $\lambda\in\C$, is the
 multivariate Laplace transform $\hat{\mathcal{L}}_f$ of $f$ evaluated at $\lambda\e\in\C^d$. 
Then Theorem \ref{thm-laplace} states that under Assumption \ref{ass-1}, for every $y\in (0,+\infty)$, there exists $\lambda_y\in (0,+\infty)$  such that
$v(y)=\lambda_y\,\hat{\L}_f(\lambda_y\,\e)$, i.e., $v$ is directly related to the Laplace transform 
$\hat{\L}_f$ of $f$, evaluated on the  diagonal $\lambda\e$. 
\begin{ex}
\label{ex-2}
For instance let $f(\x):=x_1^{\alpha_1}\cdots x_d^{\alpha_d}$ with $-1<\alpha_i\in\R$ for all $i$. Then
 \begin{eqnarray*}
 \hat{\mathcal{L}}_f(\bgamma)&=&\int_{\{\x:\x\geq0\}}\prod_{i=1}^dx_i^{\alpha_i}\,\exp(-\langle\bgamma,\x\rangle)\,d\x\\
 &=&\prod_{i=1}^d \int_{\{x_i:x_i\geq0\}}x_i^{\alpha_i}\,\exp(-\gamma_i\,x_i)\,d\x\\
 &=&\prod_{i=1}^d
 \frac{\Gamma(1+\alpha_i)}{\gamma_i^{1+\alpha_i}}\,,
 \end{eqnarray*}
 and therefore, for every $\lambda\in (0,+\infty)$:
 \[\mathcal{L}_v(\lambda)\,=\,\frac{1}{\lambda}\hat{\mathcal{L}}_f(\lambda\,\e)\,=\,
 \frac{\prod_{i=1}^d\Gamma(1+\alpha_i)}{\lambda^{1+d+\sum_{i=1}^d\alpha_i}}\,,\]
 which yields
 \[v(y)\,=\,y^{(d+\sum_{i=1}^d\alpha_i)}\,\frac{\prod_{i=1}^d\Gamma(1+\alpha_i)}{\Gamma(1+d+\sum_{i=1}^d\alpha_i)}\,,\]
 for every $y\in [0,+\infty)$.
 Again one may easily check that for every $y\in (0,\infty)$, there exists $\lambda_y\in (0,\infty)$ such that
 $v(y)=\lambda_y\,\mathcal{L}_v(\lambda_y)$.
Therefore $\lambda_y$ in Theorem 2.2 satisfies
\[v(y)\,=\,\lambda_y\,\mathcal{L}_v(\lambda_y)\,=\,\hat{\mathcal{L}}_f(\lambda_y\,\e)\,,\quad\forall y\in (0,+\infty)\,.\]
\end{ex}
\vspace{.2cm}

By additivity, the treatment of generalized polynomials
\[f(\x)\,:=\,\sum_{k=1}^sf_k\,\left(\prod_{i=1}^dx_i^{\alpha_{ki}}\right)\,,\:\mbox{(with $\alpha_{ki}>-1$ for all $k,i$)}\]
is then straightforward.  

The case of the simplex is interesting because for functions $f$  whose Laplace transform $\hat{\mathcal{L}}_f$ is 
known, one obtains directly
$\mathcal{L}_v(\lambda)$ as $\hat{\mathcal{L}}_f(\lambda\,\e)/\lambda$, i.e., with no computation involved.
Several examples of multivariate functions $f$ on $\R^d_+$
whose Laplace transform $\hat{\mathcal{L}}_f$ is readily available, are provided in \cite[\S 4]{lass-zeron}. For such functions $f$, the resulting function $\mathcal{L}_v(\lambda)$ is  explicit. 

\subsection*{Mean Value Theorem}

For the optimization problem $\P$ in \eqref{hat-v}, in addition to the optimal value $\hat{v}(z)$,
one is also interested in extracting a minimizer $\hat{\x}_z\in \tilde{K}_z$.
The counterpart for integration of ``extraction" of minimizer in optimization,
is provided by the Mean Value Theorem. 
 Indeed if $f$ is continuous and $K_y$ is compact and connected then by the Mean Value Theorem (MVT),
for every $y\in (0,+\infty)$, there exists $\x^*_y\in K_y$ such that
\begin{equation}
\label{x-y}
v(y)=\int_{K_y}f(\x)\,d\x\,=\,f(\x^*_y)\,\mathrm{vol}(K_y)
\end{equation}
to compare with $\hat{v}(z)=f(\hat{\x}_z)$ for some $\hat{\x}_z\in \tilde{K}_z$
for the optimization problem \eqref{hat-v}. Clearly, as for the extraction of 
the minimizer $\hat{\x}_z$ in optimization, MVT ``extracts" a distinguished point $\x^*_y\in K_y$ that ``explains" $v(y)$. 

In fact this statement can be made more formal after introducing 
the \emph{Maslov integral} in idempotent analysis
denoted $\int^{\oplus}$ in \cite[\S 4, p. 1485 ]{Litvinov},
that is, working in the max-plus algebra $(\oplus,\odot)$ where $a\oplus b=\max[a,b]$ and $a\odot b=a+b$. Then 
with the semi-ring $S=(\R,\oplus,\odot)$ and $B(X,S)$ the space of functions $X\to S$ that are bounded,
the idempotent analog of integration is 
\[\int^{\oplus}_{X} \phi(\x)\,d\x\,=\,\sup_{\x\in X}\phi(\x)\,,\]
for all bounded functions $\phi\in B(X,S)$; see \cite[\S 4]{Litvinov}.  So the ``$\sup$" is viewed as an ``integration" in this algebra (and 
therefore with $\mathrm{``vol"}(X)=\int^{\oplus}_X d\x=1$). Hence if $\x^*=\arg\max_{\x\in X}\phi(\x)$ then one has the MVT 
\[\phi(\x^*)\,=\,\sup_{\x\in X}\phi(\x)\,=\,\int_X^{\oplus}\phi(\x)\,d\x\,=\,\phi(\x^*)\,\int^{\oplus}_Xd\x\,\]
as for classical integration. For more details on idempotent analysis the interested reader is referred to
 e.g. \cite{Litvinov,maslov} and the many references therein.

\section{Conclusion}
We have provided  a Laplace duality framework 
for integrals on domains $K_y\subset\R^d$ parametrized by $y\in (0,+\infty)$. It mimics Legendre-Fenchel duality in optimization, and under some conditions
we have exhibited existence of a distinguished Laplace dual 
variable $\lambda_y$ that permits to replace the initial integral on the compact domain $K_y$ with an associated integral on the whole $\R^d$, with same value. In the  homogeneous case one obtains an explicit expression of $\lambda_y$  and it  would be interesting to identify other cases where such an identification is possible because then, the original integral 
can be approximated, e.g.,  by cubatures for the integral on the whole $\R^d$
with the identified specific exponential weight.



\begin{thebibliography}{las}
\bibitem{arnold}
V.~I. Arnol'd, "Normal forms of functions in the neighborhood of degenerate critical points",
\emph{Usp. Mat. Nauk}, vol 29, No 2, pp. 11--49, 1974.
\bibitem{linearity}
F. Bacelli, G. Cohen, G.J. Olsder, and J.P. Quadrat, \emph{Synchronization and Linearity: An Algebra for Discrete Event Systems}, John Wiley \& Sons Ltd., 1992
\bibitem{Cools-1}
R. Cools, "An encyclopaedia of cubature formulas", \emph{J. Complexity}, vol 19,  no. 3, pp. 445--453, 2003
\bibitem{Cools-2}
R. Cools, and A. Haegemans, "Algorithm 824: CUBPACK: A package for automatic cubature; framework description", \emph{Trans. Math. Software}, vol  29, no. 3, pp. 287--296, 2003.
\bibitem{Laplace-book}
L. Debnath, \emph{Integral Transforms and Their Applications}, CRC Press Inc., Boca Raton, Florida, 1995.
\bibitem{lasserre-book}
J.B. Lasserre, \emph{Linear and Integer Programming versus 
Linear Integration and Counting}, Springer Series in Operations Research and Financial Engineering, Springer, New York, 2009.
 \bibitem{lass-zeron}
J.B. Lasserre, and E.S. Zeron, "Solving a class of multivariable integration problems via Laplace techniques", \emph{Appl. Math. (Warsaw)}, vol 28, pp. 391--405, 2001
\bibitem{volume}
J.B. Lasserre, and E.S. Zeron, "A Laplace transform algorithm for the volume of a convex polytope", \emph{JACM}, vol 48, pp. 1126--1140, 2001.
\bibitem{survey}
J.S. Li, Z. Yang,  and Y.Z. Luo, "A review of space-object collision probability computation methods",
\emph{Astrodyn}, vol 6, pp. 95--120, 2022.
\bibitem{Litvinov}
G.L. Litvinov, "Idempotent and tropical mathematics; complexity of algorithms and interval analysis",
\emph{Comp. Math. Appl.} vol 65, pp. 1483--1496, 2013
\bibitem{maslov}
V.P. Maslov, \emph{M\'ethodes Op\'eratorielles}, Mir, Moscow, 1987.
 \bibitem{cubature}
 R. Orive, J.C. Santos-Leon, and M.G. Spalevic, "Cubature formulae for Gaussian weight: Some old and new rules",  \emph{Elec. Trans. Num. Anal.}, vol 53, pp. 426--438, 2020.
 \bibitem{schiff}
 J.~L. Schiff, \emph{The Laplace Transform: Theory and Applications}, Springer-Verlag, New York, 1999.
\bibitem{collision}
R. Serra, D. Arzelier, M. Joldes, J.B. Lasserre, A. Rondepierre, and B. Salvy, "Fast and Accurate Computation of Orbital Collision Probability for Short-Term Encounters", \emph{J. Guidance, Control \& Dynamics}, vol 39,  pp. 1--13, 2016.
\bibitem{varchenko}
A.~N. Varchenko, "Newton polyhedra and estimation of oscillating integrals, \emph{Funkts. Anal. Prilozhen.}, vol 10, No 3, pp. 13--38, 1976.
 \bibitem{vasiliev-1}
 V.A. Vasil'ev, "Asymptotic exponential integrals, Newton's diagram, and the classification of minimal points",
 Translated from \emph{Funktsional'nyi Analiz i Ego Prilozheniya},
Vol. ll, No. 3, pp. i-ii, July-September, 1977.
\bibitem{vasiliev-2}
 V.A. Vasil'ev, "Asymptotic behavior of  exponential integrals in the complex domain",
 Translated from \emph{Funktsional'nyi Analiz i Ego Prilozheniya}, Vol. 13, No. 4, pp.
1-12, October-December, 1979.
  \end{thebibliography}
 \end{document}